\documentclass[12pt,a4paper]{article}

\usepackage[
left=0.9in,
right=0.9in,
top=1in,
bottom=1in
]{geometry}

\usepackage{amsmath,amssymb,amsfonts,amsthm}
\usepackage{mathtools}
\usepackage{mathrsfs}

\usepackage{graphicx}
\usepackage{booktabs}
\usepackage{multirow}
\usepackage{threeparttable}
\usepackage{float}
\usepackage{siunitx}

\usepackage{algorithm}
\usepackage{algorithmicx}
\usepackage{algpseudocode}
\usepackage{subfig}

\usepackage{xcolor}
\usepackage{xcolor}

\usepackage[
colorlinks=true,
linkcolor=blue,
citecolor=blue,
urlcolor=blue
]{hyperref}
\usepackage{enumitem}
\usepackage{textcomp}
\usepackage[title]{appendix}

\usepackage{authblk}

\theoremstyle{plain}
\newtheorem{theorem}{Theorem}[section]

\theoremstyle{definition}
\newtheorem{definition}[theorem]{Definition}
\newtheorem{example}[theorem]{Example}

\theoremstyle{remark}

\usepackage[
left=0.9in,
right=0.9in,
top=1in,
bottom=1in
]{geometry}

\usepackage{amsmath,amssymb,amsfonts,amsthm}
\usepackage{mathtools}
\usepackage{mathrsfs}

\usepackage{graphicx}
\usepackage{booktabs}
\usepackage{multirow}
\usepackage{threeparttable}
\usepackage{float}
\usepackage{siunitx}

\usepackage{algorithm}
\usepackage{algorithmicx}
\usepackage{algpseudocode}
\usepackage{subfig}

\usepackage{xcolor}
\usepackage{xcolor}

\usepackage[
colorlinks=true,
linkcolor=blue,
citecolor=blue,
urlcolor=blue
]{hyperref}
\usepackage{enumitem}
\usepackage{textcomp}
\usepackage[title]{appendix}

\usepackage{authblk}

\begin{document}
	
	\title{\bfseries
		Tests for Increasing Convex Ordering Based on Generalized Tsallis Entropy Measures}
	
	\author[1]{Aritra Saha}
	\author[2]{Siddhartha Chakraborty\thanks{Corresponding author}}
	\author[1]{Md. Zafar Anis}
	
	\affil[1]{%
		SQC \& OR Unit\\
		Indian Statistical Institute\\
		203 Barrackpore Trunk Road\\
		Kolkata--700108, India\\
		Email: sahaaritra0@gmail.com; zafar@isical.ac.in}
	
	\affil[2]{%
		Department of Statistics\\
		University of Kalyani\\
		Kalyani, West Bengal, India\\
		Email: siddharthatdkr@gmail.com}
	\maketitle
	\begin{abstract}
		In this paper, we study several incomplete entropy measures, namely the Incomplete Weighted Cumulative Residual Entropy, the Incomplete Cumulative Residual Tsallis Entropy and its weighted version, and introduce the associated partial orderings. Their connections with certain well-known stochastic orderings are also investigated. Based on these characterizations, a class of nonparametric tests for stochastic equality against ordered alternatives is developed. The asymptotic properties of the proposed tests are derived, while their finite-sample performances are assessed through extensive Monte Carlo simulations under various alternative models and sample sizes. These tests are further compared with the test based on incomplete cumulative residual entropy proposed by Zardasht (2015).
	\end{abstract}

	\noindent
	\textbf{Keywords:}
	Mean Residual Life; Incomplete Cumulative Residual Entropy; Increasing Convex Ordering; Asymptotic Normality.
	
	\vspace{0.5em}
	
	\noindent
	\textbf{MSC (2020):} 94A17, 60E15, 62G1
	
	
	\section{Introduction}
	
	Shannon (1948) first mathematically formulated entropy as a measure of uncertainty. For a non-negative continuous random variable (rv) $X$ with probability density function (pdf) $f$, entropy is defined as 
	\begin{equation}\label{e1}
	H(X)=-\int_{0}^{\infty} f(x)\log f(x)dx,
	\end{equation}
	where log is the natural logarithm, $0\log0=0$ for computational convenience. Various generalizations of entropy have been considered in the literature over the years. One of the important generalizations is due to Rao et al. (2004) which is known as cumulative residual entropy (CRE). The CRE of a non-negative continuous rv $X$ is given by
	
	\begin{equation}\label{e2}
	\mathcal{E}(X)=-\int_0^\infty \bar F(x)\log\bar F(x)\,dx = \int_0^\infty \bar F(x) ~\Lambda(x)~dx
	\end{equation}
	where $\bar{F}(x)=P(X>x)$ is the survival function (sf) of $X$ and $\Lambda(x)$ is the cumulative hazard function. Both entropy and CRE have applications in different fields such as information theory, physics and astronomy, finance, image processing, statistics and reliability, etc. Like entropy, CRE is also studied quite extensively and numerous generalizations of CRE have been developed ever since. Rajesh and Sunoj (2019) proposed a generalization of CRE as 
	\begin{equation}\label{e3}
	\xi_{\alpha}(X)=\frac{1}{\alpha-1}\int_{0}^{+\infty}\left(\bar{F}(x)-\bar{F}^{\alpha}(x) \right)dx,\;0<\alpha\neq 1. 
	\end{equation}
	This measure is called the cumulative residual Tsallis entropy (CRTE) of order $\alpha$ and reduces to CRE when $\alpha\to1$.
	
	These uncertainty measures only consider the probabilistic information of the underlying rvs but do not take account the realizations of the rvs. Belis and Guiasu (1968) first addressed this problem by introducing weighted entropy measure which is defined as
	\begin{equation}\label{e4}
	H^w(X)=-\int_{0}^{+\infty}xf(x)\log f(x)dx,
	\end{equation}
	where the factor $x$ is the linear weight function that gives more importance to the larger values of the rv $X$. Motivated by their work, Misagh et al. (2011) introduced the weighted cumulative residual entropy (WCRE) measure which is defined as 
	\begin{equation}\label{e5}
	\xi^w(X)=-\int_{0}^{+\infty}x\bar{F}(x)\log \bar{F}(x)dx
	\end{equation}
	and Chakraborty and Pradhan (2023b) proposed the weighted cumulative residual Tsallis entropy (WCRTE) as 
	
	\begin{equation}\label{e6}
	\xi_{\alpha}^w(X)=\frac{1}{\alpha-1}\int_{0}^{+\infty}x\left(\bar{F}(x)-\bar{F}^{\alpha}(x) \right)dx,\;0<\alpha\neq 1. 
	\end{equation}
	When $\alpha\to1$, it reduces to WCRE. Various weighted entropy measures have been developed by many authors with diverse applications. Interested readers may refer to Mirali and Baratpour (2017, 2022), Chakraborty and Pradhan (2023a, 2024), Kayal
	and Balakrishnan (2023), Toomaj and Di Crescenzo (2020) and the references therein.\\
	
	


	Zardasht (2015) introduced the concept of incomplete CRE function and used it to develop a test for increasing convex order. The ``upper'' incomplete CRE (ICRE) function is defined as
	\begin{equation} \label{Eq. Zardasht (2015)}
	\mathcal{E}(X,t) = \int_t^\infty \bar F(x) ~\Lambda(x)~dx,
	\end{equation}
	where $\mathcal{E}(X,0)=\mathcal{E}(X)$. Stochastic orderings play an important role in probability, statistics, reliability theory, actuarial science, economics, and risk management, as they provide useful tools for comparing random variables and lifetime distributions; see Berrendero and Cárcamo (2011) and Shaked and Shanthikumar (2007). In particular, the increasing convex order is widely used in reliability analysis and actuarial studies for comparing variability, ageing properties, and risk characteristics of systems and portfolios. Such orderings also arise naturally in queueing theory, survival analysis, and decision theory, where comparisons between uncertain models are of primary interest.
	
	Inspired by the work of Zardasht (2015) and the broad applicability of stochastic orderings, we study several incomplete entropy measures and develop corresponding testing procedures for increasing convex order alternatives.

	
	The remainder of the paper is organised as follows. \hyperref[section prelim and problem statement]{Section 2} presents the necessary preliminaries, corresponding hypothesis testing framework and the problem formulation. In \hyperref[section iwcre]{Section 3}, a new stochastic ordering based on the IWCRE is introduced, and its relationships with the usual stochastic order and the increasing convex order are investigated. Based on this ordering, a test statistic for testing stochastic equality against increasing convex ordered alternatives is developed. Similar developments corresponding to ICRTE and IWCRTE are presented in Sections~\ref{section icrte} and \ref{section iwcrte}, respectively. \hyperref[section simulation_study]{Section 6} contains extensive Monte Carlo simulation studies along with comparative analyses of the proposed tests. Finally, concluding remarks and possible future directions are provided in \hyperref[section Conclusion]{Section 7}.

	\section{Preliminary Concepts and Problem Statement}
	\label{section prelim and problem statement}
	
	In this section, we discuss the usual stochastic order, increasing convex order and related testing problem. Consider the following definitions.
	
	\begin{definition} \label{def:stochastic-order} 
		The random variable $X$ is said to be smaller than $Y$ in the usual stochastic order ($X \le_{st} Y$) if 
		\begin{equation}
		\bar F(t) \le \bar G(t)\quad \text{for all } t\ge0.
		\end{equation}
		
	\end{definition}

	\begin{definition} \label{def:icx-order}
		The random variable $X$ is said to be smaller than $Y$ in the increasing convex order ($X \le_{icx} Y$) if 
		\begin{equation}
		\int_t^{\infty} \bar F(x)\,dx \le 
		\int_t^{\infty} \bar G(x)\,dx
		\quad \text{for all } t\ge 0.
		\end{equation}
	\end{definition}
	
	The objective is to test the following hypothesis:
	\begin{equation}
	H_0: X =_{st} Y \quad \text{vs} \quad H_1: X \le_{icx} Y, \;\; X \ne_{st} Y.
	\end{equation}
	Zardasht (2015) defined a new ordering called incomplete cumulative residual entropy order and based on this, they proposed a test statistic for the above mentioned problem. 
	
	\begin{definition}
		$X$ is said to be less than $Y$ in incomplete cumulative residual entropy  (denoted $X \le_{ine} Y$) if for all $t \ge 0$,
		\[
		\mathcal{E}(X,t) \le \mathcal{E}(Y,t).
		\]
	\end{definition}

	\noindent It has been established in the literature that important connections exist between the ICRE-based ordering and other well-known stochastic orders. In particular, Zardasht (2015) showed that if $X \le_{st} Y$, then it follows that $X \le_{ine} Y$ and if $X \le_{icx} Y$, then $X \le_{ine} Y$. These findings indicate that the ICRE-based ordering is consistent with, and in fact contained within, two fundamental stochastic orders.
	Consequently, incomplete cumulative residual (ICRE) entropy provides a valid and interpretable framework for comparing lifetime distributions.

	Since $X \le_{icx} Y$ implies $X \le_{ine} Y$, for random variables $X$ and $Y$ with continuous distribution functions $F$ and $G$, Zardasht (2015) introduced the following discrepancy measure:
	\begin{equation}
	\Delta(F,G) = \mathcal{E}(Y) - \mathcal{E}(X).
	\end{equation}
	
	Under the null hypothesis $H_0$, the random variables $X$ and $Y$ are identically distributed, and hence
	\begin{equation}
	\Delta(F,G) = 0.
	\end{equation}
	
	Under the alternative hypothesis $H_1$, we have $X \le_{icx} Y$, which implies
	\begin{equation}
	\mathcal{E}(X,t) \le \mathcal{E}(Y,t), \quad \text{for all } t \ge 0,
	\end{equation}
	and consequently,
	\begin{equation}
	\Delta(F,G) > 0.
	\end{equation}
	
	Therefore, one can reject the null hypothesis when $\Delta(F,G)$ is sufficiently large.

	\section{Incomplete Weighted Cumulative Residual Entropy}
	\label{section iwcre}
	
	The ‘‘upper’’ incomplete weighted cumulative residual entropy (IWCRE) function is defined as follows:
	
	\begin{equation} \label{Eq. WCRE general equation}
	\xi^{w}(X,t)
	= -\int_{t}^{\infty}
	x\,\bar{F}(x)\,\log\bar{F}(x)\,dx.
	\end{equation}
	
	Observe that $\xi^{w}(X,t)$ can be written as 
	
	\begin{equation}
	\xi^{w}(X,t)
	=
	\mathbb{E}\bigl(\psi^{w}(t,X)I(X>t)\bigr),
	\end{equation}
	where
	\begin{equation} \label{Eq. psi^{w}(t,x)}
	\psi^{w}(t,x)
	=
	-\int_{t}^{x}
	u\log \bar F(u)\,du.
	\end{equation}

	\begin{proof}
		From Eq.~\eqref{Eq. WCRE general equation}, we have
		\[
		\xi^{w}(X,t)
		=
		-\int_{t}^{\infty}
		x\,\bar F(x)\log \bar F(x)\,dx.
		\]
		
		Substituting $\bar F(x)$ by $\int_x^\infty dF(y)$, we obtain
		\begin{align*}
		\int_t^\infty x\,\bar F(x)\,g(x)\,dx
		&=
		\int_t^\infty
		x
		\left(
		\int_x^\infty dF(y)
		\right)
		g(x)\,dx \\
		&=
		\int_t^\infty
		\int_x^\infty
		x\,g(x)\,dF(y)\,dx,
		\end{align*}
		where $g$ is a nonnegative measurable function given by
		\[
		g(x)=-\log \bar F(x).
		\]
		
		By Fubini's theorem,
		\begin{align*}
		\int_t^\infty
		\int_x^\infty
		x\,g(x)\,dF(y)\,dx
		&=
		\int_t^\infty
		\left(
		\int_t^y x\,g(x)\,dx
		\right)dF(y).
		\end{align*}
		
		Therefore,
		\begin{align*}
		\xi^{w}(X,t)
		&=
		\int_t^\infty
		\psi^{w}(t,y)\,dF(y) \\
		&=
		\mathbb{E}\bigl(\psi^{w}(t,X)I(X>t)\bigr),
		\end{align*}
		where
		\[
		\psi^{w}(t,y)
		=
		-\int_t^y x\log \bar F(x)\,dx.
		\]
		
		Hence proved
	\end{proof}
	Now we define IWCRE order and study its relation with usual stochastic order and increasing convex order.
	
	\begin{definition} \label{Definition IWCRE order}
		\(X\) is said to be less than \(Y\) in incomplete weighted cumulative residual entropy (IWCRE) (denoted by \(X \leq_{iwcre} Y \)) if for all $t \geq 0$, 
		\begin{equation}
		\xi^{w}(X,t) \leq \xi^{w}(Y,t).
		\end{equation}
	\end{definition}
	
	The following results give the relationship between the IWCRE ordering and the orderings described in Definition \ref{def:stochastic-order} and Definition \ref{def:icx-order}. First, the
	theorem below shows that the usual stochastic order implies the IWCRE order.
	
	\begin{theorem} \label{Theorem WCRE for stochastic order}
		If $X \leq_{st} Y$, then \(X \leq_{iwcre} Y \).
	\end{theorem}
	
	\begin{proof}
		Define $h(x)=\psi^{w}(t,x)I(x>t)$. Since $h(x)=0$ for $x\le t$, we only consider the case for $x>t$, where 
		\begin{equation} \label{Eq. h(x) of WCRE}
		h(x) = \psi^{w}(t,x) = -\int_t^x u\log \bar F(u)\,du.
		\end{equation}

		Differentiating with respect to $x$, we obtain
		$h'(x)=-x\log \bar F(x)$. Now, since $0<\bar F(x)\le1$, it follows that $\log \bar F(x)\le0$. Therefore,
		$h'(x)\ge0$, for all $x>t$. Hence, $h(x)$ is increasing in $x$.
		
		Thus, for any given $t$, the function
		$h(x)$
		is increasing in $x$. The result now follows from the characterization of the usual stochastic order i.e.,
		$X\leq_{st}Y$
		if and only if
		$\mathbb{E}[h(X)]\leq \mathbb{E}[h(Y)]$
		for all increasing functions $h$ for which expectations exist.
	\end{proof}

	The following example demonstrates that the converse of Theorem \ref{Theorem WCRE for stochastic order} does not hold, in general.
	
	\begin{example}
		Let $X$ be distributed as Weibull with survival function
		\begin{equation} \label{Eq. X:Weibull(a,b)}
		\bar{F}(x)=e^{-\left(\frac{x}{b}\right)^{a}},
		\qquad x>0,\;a,b>0,
		\end{equation}
		and $Y$ be distributed as Pareto with survival function
		\begin{equation} \label{Eq. Y:Pareto(c,d)}
		\bar{G}(x)=\left(\frac{d}{x+d}\right)^{c},
		\qquad x>0,\;c>2,\;d>0.
		\end{equation}
		
		Then,
		\[
		\xi^{w}(X,t)
		=
		\frac{b^{2}}{a}
		\Gamma\left(
		\frac{a+2}{a},
		\left(\frac{t}{b}\right)^{a}
		\right),
		\]
		and
		\[
		\xi^{w}(Y,t)
		=
		\frac{cd^{2}}{(c-2)^{2}}
		\Gamma\left(
		2,
		(c-2)\log\left(\frac{t+d}{d}\right)
		\right)
		-
		\frac{cd^{2}}{(c-1)^{2}}
		\Gamma\left(
		2,
		(c-1)\log\left(\frac{t+d}{d}\right)
		\right),
		\]
		where $\Gamma(a,t) = \int_{t}^{\infty} y^{a-1}e^{-y}\,dy$ is the incomplete gamma function.
		
		For $a=2$, $c=3$ and $b=d=1$, it can be seen that
		\[
		\xi^{w}(X,t)
		=
		\frac{1}{2}\Gamma(2,t^{2})
		\leq
		3\Gamma\bigl(2,\log(1+t)\bigr)
		-
		\frac{3}{4}\Gamma\bigl(2,2\log(1+t)\bigr)
		=
		\xi^{w}(Y,t).
		\]
		
		On the other hand, we have
		\[
		\bar{G}(0.5)
		=
		\left(\frac{1}{1.5}\right)^{3}
		=
		0.296
		<
		0.779
		=
		e^{-0.25}
		=
		\bar{F}(0.5),
		\]
		which shows that $\bar{F}(t)\leq \bar{G}(t)$ does not hold for all $t\geq0$.
	\end{example}
	
	\begin{theorem} \label{Theorem WCRE for convex order}
		If $X \leq_{icx} Y$, then \(X \leq_{iwcre} Y \).
	\end{theorem}
	
	\begin{proof}
		Double differentiating $h(x)$ in Eq.~\eqref{Eq. h(x) of WCRE} with respect to $x$, we get 
		$$h''(x) = -\log \bar F(x) + x\frac{f(x)}{\bar F(x)}.$$
		Both the terms $-\log \bar F(x)\ge0$ and $x\frac{f(x)}{\bar F(x)}\ge0$ for all $x>t$. Therefore, $h''(x)\ge0$, for all $x>t$. Hence, for any given $t$, the function
		$h(x)$ is convex in $x$. The result now follows from the characterization of the convex order, i.e., $X\leq_{icx}Y$ if and only if $\mathbb{E}[h(X)]\leq \mathbb{E}[h(Y)]$ for all convex functions $h$ for which the expectations exist.
	\end{proof}
	
	The next example shows that the converse of Theorem \ref{Theorem WCRE for convex order} is not, in general, true.
	\begin{example} \label{WCRE convex order counter example}
		We consider the distributions studied in Zardasht (2015). Let $X$ have Weibull distribution with survival function $\bar{F}(t)=e^{-t^{5}},~t>0,$ and let $Y$ have Gompertz distribution with survival function $\bar{G}(t)=e^{-e^{t}+1},~t>0.$
		
		Figures~\ref{convex ordering graph} and \ref{IWCRE graph} depict the plots of $\int_{t}^{\infty}\bar{F}(x)\,dx,\int_{t}^{\infty}\bar{G}(x)\,dx$ and $\xi^{w}(X,t),~\xi^{w}(Y,t),$ respectively. The plots show that $\xi^{w}(X,t)\leq\xi^{w}(Y,t)$ for all $t\geq0.$ On the other hand, there exists some $t_{0}>0$ such that $\int_{t_{0}}^{\infty}\bar{F}(x)\,dx>\int_{t_{0}}^{\infty}\bar{G}(x)\,dx.$
	\end{example}
	

	\subsection{Testing stochastic equality against increasing convex ordering using WCRE function}
	
	We want to test
	\begin{equation}
	H_0: X =_{st} Y \quad \text{vs} \quad H_1: X \le_{icx} Y, \;\; X \ne_{st} Y.
	\end{equation}
	
	Since $X \le_{icx} Y$ implies $X \le_{iwcre} Y$, for random variables $X$ and $Y$ with continuous distribution functions $F$ and $G$, consider the discripency measure
	\begin{equation}
	\Delta^{w}(F,G) = \xi^w(Y)-\xi^w(X).
	\end{equation}
	Clearly, $\Delta^{w}(F,G)\geq0$ and equality holds under $H_0$. So, the null hypothesis can be rejected when $\Delta^{w}(F,G)$ is large. Now, we need to estimate $\Delta^{w}(F,G)$ from given data sets. Chakraborty and Nanda (2025) proposed an estimator for $\xi^w(X)$ as
	\[
	\hat{\xi}(X) = -\frac{1}{2n} \sum_{i=1}^{n} X_{(i)}^{2}~ J_{2}(\frac{i}{n+1}),
	\]
	where $J_2(u)=1+\log(1-u),\;\forall\;0\leq u <1$. They showed that $\hat{\xi}(X)$ is consistent and asymptotically normally distributed with mean $\xi^w(X)$ and variance $\frac{\sigma_{X}^2(J_2)}{n}$.
	\begin{equation*}
	\sigma_{X}^2(J_2) = 2~\int_{0}^{\infty}\int_{y}^{\infty} x~y~F(y)~\bar{F}(x)~J_{2}(F(x))~ J_2(F(y))~dx~dy.
	\end{equation*}
	Using this estimator, we can estimate $\Delta^{w}(F,G)$ as 
	\[
	\hat{\Delta}^{w}(F_n,G_m) = \hat{\xi}(Y) - \hat{\xi}(X),
	\]
	where
	\[
	\hat{\xi}(Y) = -\frac{1}{2m} \sum_{j=1}^{m} Y_{(j)}^{2}~ J_{2}(\frac{j}{m+1}).
	\]
	From the asymptotic normality of $\hat{\xi}(X)$ and $\hat{\xi}(X)$, we have
	\begin{equation}
	\sqrt{N}[\hat{\Delta}^{w}(F_n,G_m)-\Delta^{w}(F,G)] \to^{d}~N(0,\sigma^2(J_2))
	\end{equation}
	where
	\begin{equation*}
	\sigma^2(J_2) = \frac{\sigma_{X}^2(J_2)}{\rho}+\frac{\sigma_{Y}^2(J_2)}{1-\rho}.
	\end{equation*}
	
	An estimator of $\sigma_{X}^2(J_2)$ can be obtained as
	
	\begin{equation}
	\begin{aligned}
	\hat{\sigma}_{X}^{2}(J_2) 
	&= 2 \sum_{j=1}^{n-2}\sum_{i=j+1}^{n-1} 
	\frac{j}{n} \left(1- \frac{i}{n}\right) 
	J_2 \left(\frac{i}{n+1}\right) 
	J_2\left(\frac{j}{n+1}\right) \\
	&\quad \times \left(X_{(i+1)} - X_{(i)}\right)
	\left(X_{(j+1)} - X_{(j)}\right).
	\end{aligned}
	\end{equation}
	
	Under null hypothesis, 
	\begin{equation*}
	\frac{\hat{\Delta}^{w}(F_n,G_m)}{\sqrt{\frac{\hat{\sigma}_{X}^{2}(J_2)}{n}+\frac{\hat{\sigma}_{Y}^{2}(J_2)}{m}}} \to^d N(0,1).
	\end{equation*}

	\section{Incomplete cumulative residual Tsallis entropy}
	\label{section icrte} 
	
	The ‘‘complete’’ CRTE function $T_{\alpha}(X)$ can be generalized to the incomplete CRTE function $T_{\alpha}(X,t)$ such that $T_{\alpha}(X,0)= T_{\alpha}(X)$. This ‘‘upper’’ incomplete CRTE (ICRTE) function is defined as follows:
	
	\begin{equation} \label{Eq. Incomplete CRTE equation RajeshSunoj}
	T_{\alpha}(X,t) =\frac{1}{\alpha-1}\int_{t}^\infty\big(\bar F(x)-\bar F^\alpha(x)\big)\,dx.
	\end{equation}

	It can be shown that
	\begin{equation}
	T_{\alpha}(X,t)
	=
	\mathbb{E}\bigl(\psi_{\alpha}(t,X)I(X>t)\bigr),
	\end{equation}
	where
	\begin{equation} \label{Eq. psi {alpha}(t,x)}
	\psi_{\alpha}(t,x) = \frac{1}{\alpha-1} \int_{t}^{x} \bigl(1-\bar F^{\alpha-1}(y)\bigr)\,dy.
	\end{equation}

	\begin{proof}
		We can write Eq.~\eqref{Eq. Incomplete CRTE equation RajeshSunoj} as  
		\[
		T_{\alpha}(X,t)
		=
		\frac{1}{\alpha-1}
		\int_{t}^{\infty}
		\bar F(x)\bigl(1-\bar F^{\alpha-1}(x)\bigr)\,dx.
		\]
		
		Substituting $\bar F(x)$ by $\int_x^\infty dF(y)$, we have 
		\begin{align*}
		\int_t^\infty \bar F(x)\,g(x)\,dx
		&=
		\int_t^\infty
		\left(
		\int_x^\infty dF(y)
		\right)
		g(x)\,dx \\
		&=
		\int_t^\infty
		\int_x^\infty
		g(x)\,dF(y)\,dx,
		\end{align*}
		where $g$ is a nonnegative measurable function given by
		\[
		g(x)=1-\bar F^{\alpha-1}(x).
		\]
		
		By Fubini's theorem,
		\begin{align*}
		\int_t^\infty
		\int_x^\infty
		g(x)\,dF(y)\,dx
		&=
		\int_t^\infty
		\left(
		\int_t^y g(x)\,dx
		\right)dF(y).
		\end{align*}
		
		So $T_{\alpha}(X,t)$ can be written as
		\begin{align*}
		T_{\alpha}(X,t)
		&=
		\int_t^\infty
		\psi_{\alpha}(t,y)\,dF(y) \\
		&=
		\mathbb{E}\bigl(\psi_{\alpha}(t,X)I(X>t)\bigr),
		\end{align*}
		where
		\[
		\psi_{\alpha}(t,y)
		=
		\frac{1}{\alpha-1}
		\int_t^y
		\bigl(1-\bar F^{\alpha-1}(x)\bigr)\,dx.
		\]
	\end{proof}
	
	Next, we define ICRTE ordering between two rvs. Consider the following definition.
	\begin{definition}
		\(X\) is said to be less than \(Y\) in incomplete cumulative residual tsallis entropy (ICRTE) (denoted by \(X \leq_{icrte} Y \)) if for all $t \geq 0$, 
		\begin{equation}
		T_{\alpha}(X,t) \leq T_{\alpha}(Y,t).
		\end{equation}
	\end{definition}
	The following results give the relationship between the ICRTE ordering and the orderings described in Definition \ref{def:stochastic-order} and Definition \ref{def:icx-order}. First, the theorem below shows that the usual stochastic order implies the ICRTE order.
	\begin{theorem} \label{Theorem CRTE for stochastic order}
		If $X \leq_{st} Y$, then \(X \leq_{icrte} Y \).
	\end{theorem}

	\begin{proof}
		Define $h_{1}(x)=\psi_{\alpha}(t,x)I(x>t)$. Since $h_{1}(x)=0$ for $x\le t$, we only consider the case for $x>t$, where
		\begin{equation} \label{Eq. h1(x) of CRTE}
		h_{1}(x)
		=
		\psi_{\alpha}(t,x)
		=
		\frac{1}{\alpha-1}\int_t^x \bigl(1-\bar F^{\alpha-1}(y)\bigr)\,dy.
		\end{equation}
		
		Differentiating with respect to $x$, we obtain $h_{1}'(x)
		=
		\frac{1}{\alpha-1}\bigl(1-\bar F^{\alpha-1}(x)\bigr).$ Now, since $0<\bar F(x)\le1$, we have $0<\bar F^{\alpha-1}(x)\le1$.
		Therefore, $1-\bar F^{\alpha-1}(x)\ge0.$ Hence, $h_{1}'(x)\ge0,$ for all $x>t$. Thus, for any given $t$, the function
		$h_{1}(x)$ is increasing in $x$. The result now follows from the characterization of the usual stochastic order, i.e.,
		$X\leq_{st}Y$
		if and only if
		$\mathbb{E}[h_{1}(X)]\leq \mathbb{E}[h_{1}(Y)]$
		for all increasing functions $h_{1}$ for which the expectations exist.
	\end{proof}
	
	The next example establishes that the converse statement of Theorem~\ref{Theorem CRTE for stochastic order} need not hold in general.
	
	\begin{example}
		Consider the distributions introduced in Eq.~\eqref{Eq. X:Weibull(a,b)} and Eq.\eqref{Eq. Y:Pareto(c,d)}, where $X \sim \mathrm{Weibull}(a,b)$ and $Y \sim \mathrm{Pareto}(c,d)$. For the Pareto distribution, the parameter \(c\) is assumed to satisfy \(c>1\).
		
		Then,
		\begin{equation*}
		T_{\alpha}(X,t)
		=
		\frac{b}{a(\alpha-1)}
		\left[
		\Gamma\left(
		\frac1a,
		\left(\frac{t}{b}\right)^a
		\right)
		-
		\alpha^{-1/a}
		\Gamma\left(
		\frac1a,
		\alpha\left(\frac{t}{b}\right)^a
		\right)
		\right],
		\end{equation*}
		and
		\begin{equation*}
		T_{\alpha}(Y,t)
		=
		\frac{d}{\alpha-1}
		\left[
		\frac{1}{c-1}
		\left(
		\frac{d}{t+d}
		\right)^{c-1}
		-
		\frac{1}{\alpha c-1}
		\left(
		\frac{d}{t+d}
		\right)^{\alpha c-1}
		\right],
		\quad
		c>1,\;
		\alpha c>1.
		\end{equation*}
		
		where $\Gamma(a,t) = \int_t^\infty y^{a-1}e^{-y}\,dy$ is the incomplete gamma function.
		
		For $a=c=\alpha=2$ and $b=d=1$, it can be seen that
		\[
		T_{2}(X,t)
		=
		\frac12
		\left[
		\Gamma\left(\frac12,t^2\right)
		-
		\frac{1}{\sqrt2}
		\Gamma\left(\frac12,2t^2\right)
		\right]
		\leq
		\frac{1}{1+t}
		-
		\frac{1}{3(1+t)^3}
		=
		T_{2}(Y,t).
		\]
		
		On the other hand, we have
		\[
		\bar{G}(0.5)
		=
		\left(\frac{1}{1.5}\right)^2
		=
		0.444
		<
		0.779
		=
		e^{-0.25}
		=
		\bar{F}(0.5),
		\]
		which shows that $\bar{F}(t)\leq \bar{G}(t)$ does not hold for all $t\geq0$.
	\end{example}
	
	The following theorem shows that increasing convex order implies ICRTE order.
	\begin{theorem} \label{Theorem CRTE for convex order}
		If $X \leq_{icx} Y$, then \(X \leq_{icrte} Y \).
	\end{theorem}

	\begin{proof}Double differentiating $h_{1}(x)$ in Eq.~\eqref{Eq. h1(x) of CRTE} with respect to $x$, we get\[h_{1}''(x)=f(x)~\bar F^{\alpha-2}(x).\]Since $f(x)\ge0$ and $\bar F(x)>0$, it follows that $h_{1}''(x)\ge0$, for all $x>t$. Hence, for any given $t$, the function $h_{1}(x)$ is convex in $x$. The result now follows from the characterization of the convex order, i.e.,$X\leq_{icx}Y$ if and only if $\mathbb{E}[h_{1}(X)]\leq \mathbb{E}[h_{1}(Y)]$ for all convex functions $h_{1}$ for which the expectations exist.\end{proof}

	However, the converse of the above theorem does not hold in general. Consider again the Weibull and Gompertz distributions introduced in Example~\ref{WCRE convex order counter example}. Figures~\ref{convex ordering graph} and \ref{ICRTE graph} illustrate the plots of $\int_{t}^{\infty}\bar{F}(x)\,dx,\int_{t}^{\infty}\bar{G}(x)\,dx$ and $T_{\alpha}(X,t),~T_{\alpha}(Y,t),$ respectively. The figures indicate that $T_{\alpha}(X,t)\leq T_{\alpha}(Y,t)$ for all $t\geq0.$ Nevertheless, as observed from the curves in Figure~\ref{convex ordering graph}, there exists some $t_{0}>0$ such that $\int_{t_{0}}^{\infty}\bar{F}(x)\,dx>\int_{t_{0}}^{\infty}\bar{G}(x)\,dx.$\\


	
	
	

	\subsection{Testing stochastic equality against increasing convex ordering using CRTE function}

	The objective is to test the following hypothesis:
	Here we consider the following discrepancy measure for the testing problem. Let
	\begin{equation}
	\Delta_{\alpha}(F,G) = T_{\alpha}(Y) - T_{\alpha}(X).
	\end{equation}
	We will reject null hypothesis for large values of $\Delta_{\alpha}(F,G)$.      
	The empirical version of $\Delta_{\alpha}(F,G)$ can be written as 
	
	\begin{equation}
	\hat{\Delta}_{\alpha}(F_n, G_m) = \hat{T}_{\alpha}(Y) -\hat{T}_{\alpha}(X).
	\end{equation}
	Zardasht (2019) proposed the estimator $\hat{T}_{\alpha}(X)$ as
	\begin{equation}
	\hat{T}_{\alpha}(X) = \frac{1}{n} \sum_{i=1}^{n} X_{(i)}~ J_{2 \alpha}\left(\frac{i}{n}\right),
	\end{equation}
	where $J_{2\alpha}(u)=\frac{1}{\alpha-1}(1-\alpha(1-u)^{\alpha-1})$ and proved that $\sqrt{N}(\hat{T}_{\alpha}(X)-T_{\alpha}(X)_)\to^d N(0,\sigma^{2}_X(J_{2 \alpha}))$. The formula for the variance is given as
	\begin{equation*}
	\sigma^{2}(J_{2 \alpha}) = \frac{\sigma^{2}_X(J_{2 \alpha})}{\rho} + \frac{\sigma^{2}_Y(J_{2 \alpha})}{1-\rho}.
	\end{equation*}\textbf{}
	So we get, 
	\begin{equation*}
	\sqrt{N} (\hat{\Delta}_{\alpha}(F_n, G_m)-\Delta_{\alpha}(F,G)) \to^d N(0, \sigma^{2}(J_{2 \alpha})),
	\end{equation*}
	
	where
	\begin{equation*}
	\sigma^{2}(J_{2 \alpha}) = \frac{\sigma^{2}_X(J_{2 \alpha})}{\rho} + \frac{\sigma^{2}_Y(J_{2 \alpha})}{1-\rho}.
	\end{equation*}
	
	Note that, $\sigma^{2}_X(J_{2 \alpha})$ can be estimated as 
	
	\begin{equation}
	\begin{aligned}
	\hat{\sigma}_X^{2}(J_{2\alpha})
	&= 2 \sum_{j=2}^{n-1}\sum_{i=1}^{j-1}
	\frac{i}{n} \left(1- \frac{j}{n}\right)
	J_{2\alpha}\!\left(\frac{i}{n}\right)
	J_{2\alpha}\!\left(\frac{j}{n}\right) \\
	&\quad \times (X_{(i+1)} - X_{(i)})
	(X_{(j+1)} - X_{(j)}).
	\end{aligned}
	\end{equation}
	The asymptotic distribution of the discrepancy measure under $H_0$ is
	\begin{equation*}
	\frac{\hat{\Delta}_{\alpha}(F_n, G_m)}{\sqrt{\frac{\hat{\sigma}_{X}^{2}(J_{2 \alpha})}{n}+\frac{\hat{\sigma}_{Y}^{2}(J_{2 \alpha})}{m}}} \to^d N(0,1).
	\end{equation*}

	\section{Incomplete Weighted Cumulative Residual Tsallis Entropy}
	\label{section iwcrte}
	
	The ‘‘upper'' incomplete WCRTE (IWCRTE) function is defined as follows:
	
	\begin{equation} \label{Eq. WCRTE General "t" equation}
	\xi_{\alpha}^w(X,t)
	= \frac{1}{\alpha-1}\int_{t}^\infty x\big(\bar{F}(x) - \bar{F}^{\alpha}(x)\big)\,dx,
	\qquad 0<\alpha\neq 1.
	\end{equation}
	It can be shown that 
	\begin{equation}
	\xi_{\alpha}^w(X) = \mathbb{E}\bigl(\psi_{\alpha}^{w}(X)\bigr),
	\end{equation}
	where 
	\begin{equation} \label{Eq. psi {alpha}^{w}(X)}
	\psi_{\alpha}^{w}(X) = \frac{1}{\alpha-1}
	\int_{0}^{t} x \bigl(1-\bar F^{\alpha-1}(x)\bigr)\,dx.
	\end{equation}
	
	\begin{proof}
		We have,  
		\[
		\xi_{\alpha}^w(X)
		=
		\frac{1}{\alpha-1}
		\int_{0}^{\infty}
		x\,\bar F(x)\bigl(1-\bar F^{\alpha-1}(x)\bigr)\,dx.
		\]
		
		Substituting $\bar F(x)$ by $\int_{x}^{\infty} dF(t)$, we have 
		\begin{align*}
		\int_0^\infty x\,\bar F(x)\,g(x)\,dx
		&=
		\int_0^\infty
		x
		\left(
		\int_x^\infty dF(t)
		\right)
		g(x)\,dx \\
		&=
		\int_0^\infty
		\int_x^\infty
		x\,g(x)\,dF(t)\,dx,
		\end{align*}
		where $g$ is a nonnegative measurable function given by
		\[
		g(x)=1-\bar F^{\alpha-1}(x).
		\]
		
		By Fubini's theorem,
		\begin{align*}
		\int_0^\infty
		\int_x^\infty
		x\,g(x)\,dF(t)\,dx
		&=
		\int_0^\infty
		\left(
		\int_0^t x\,g(x)\,dx
		\right)dF(t).
		\end{align*}
		
		So $\xi_{\alpha}^w(X)$ can be written as
		\begin{align*}
		\xi_{\alpha}^w(X)
		&=
		\int_{0}^{\infty}
		\psi_{\alpha}^{w}(t)\,dF(t) \\
		&=
		\mathbb{E}\bigl(\psi_{\alpha}^{w}(X)\bigr),
		\end{align*}
		where
		\[
		\psi_{\alpha}^{w}(t)
		=
		\frac{1}{\alpha-1}
		\int_{0}^{t}
		x\bigl(1-\bar F^{\alpha-1}(x)\bigr)\,dx.
		\]
		Hence proved.\\
		
		Next, we define IWCRTE order and discuss its relation with stochastic and increasing convex ordering.
	\end{proof}
	\begin{definition}
		\(X\) is said to be less than \(Y\) in incomplete weighted cumulative residual tsallis entropy (denoted by \(X \leq_{iwcrte} Y \)) if for all $t \geq 0$, 
		\begin{equation}
		T_{\alpha}^{w}(X,t) \leq T_{\alpha}^{w}(Y,t).
		\end{equation}
	\end{definition}
	The following results give the relationship between the IWCRTE ordering and the orderings described in Definition \ref{def:stochastic-order} and Definition \ref{def:icx-order}. First, the theorem below shows that the usual stochastic order implies the IWCRTE order.
	\begin{theorem} \label{Theorem WCRTE for stochastic order}
		If $X \leq_{st} Y$, then \(X \leq_{iwcrte} Y \).
	\end{theorem}
	
	\begin{proof}
		Define $h_{2}(x)=\psi_{\alpha}^{w}(t,x)I(x>t)$. Since $h_{2}(x)=0$ for $x\le t$, we only consider the case for $x>t$, where
		\begin{equation} \label{Eq. h2(x) of WCRTE}
		h_{2}(x)
		=
		\psi_{\alpha}^{w}(t,x)
		=
		\frac{1}{\alpha-1}
		\int_t^x u\bigl(1-\bar F^{\alpha-1}(u)\bigr)\,du.
		\end{equation}
		
		Differentiating with respect to $x$, we obtain $h_{2}'(x)
		=
		\frac{x}{\alpha-1}
		\bigl(1-\bar F^{\alpha-1}(x)\bigr).$ Now, since $0<\bar F(x)\le1$, we have $0<\bar F^{\alpha-1}(x)\le1$. Therefore, $1-\bar F^{\alpha-1}(x)\ge0.$ Hence, $h_{2}'(x)\ge0,$ for all $x>t$. Thus, for any given $t$, the function $h_{2}(x)$ is increasing in $x$. The result now follows from the characterization of the usual stochastic order, i.e., $X\leq_{st}Y$ if and only if $\mathbb{E}[h_{2}(X)]\leq \mathbb{E}[h_{2}(Y)]$ for all increasing functions $h_{2}$ for which the expectations exist.
	\end{proof}
	
	The example provided below shows that the converse of Theorem~\ref{Theorem CRTE for stochastic order} is not true, in general.

	\begin{example}
		Take the same Weibull and Pareto distributions whose survival functions are specified in Eq.~\eqref{Eq. X:Weibull(a,b)} and Eq.~\eqref{Eq. Y:Pareto(c,d)}.
		
		Then,
		\begin{equation*}
		\xi_{\alpha}^{w}(X,t)
		=
		\frac{b^{2}}{a(\alpha-1)}
		\left[
		\Gamma\left(
		\frac{2}{a},
		\left(\frac{t}{b}\right)^a
		\right)
		-
		\alpha^{-2/a}
		\Gamma\left(
		\frac{2}{a},
		\alpha\left(\frac{t}{b}\right)^a
		\right)
		\right],
		\end{equation*}
		and
		\begin{equation*}
		\begin{aligned}
		\xi_{\alpha}^{w}(Y,t)
		=
		\frac{d^{2}}{\alpha-1}
		\Bigg[
		&
		\frac{1}{c-2}
		\left(
		\frac{d}{t+d}
		\right)^{c-2}
		-
		\frac{1}{c-1}
		\left(
		\frac{d}{t+d}
		\right)^{c-1}
		\\[1ex]
		&
		-
		\frac{1}{\alpha c-2}
		\left(
		\frac{d}{t+d}
		\right)^{\alpha c-2}
		+
		\frac{1}{\alpha c-1}
		\left(
		\frac{d}{t+d}
		\right)^{\alpha c-1}
		\Bigg],
		\end{aligned}
		\quad
		c>2,\;
		\alpha c>2.
		\end{equation*}
		
		where $\Gamma(a,t) = \int_t^\infty y^{a-1}e^{-y}\,dy$ is the incomplete gamma function.
		
		For $a=\alpha=2$, $c=3$ and $b=d=1$, it can be seen that
		\[
		\xi_{2}^{w}(X,t)
		=
		\frac12
		\left[
		\Gamma(1,t^2)
		-
		\frac12\Gamma(1,2t^2)
		\right]
		\leq
		\frac1{1+t}
		-
		\frac1{2(1+t)^2}
		-
		\frac1{4(1+t)^4}
		+
		\frac1{5(1+t)^5}
		=
		\xi_{2}^{w}(Y,t).
		\]
		
		On the other hand, we have
		\[
		\bar{G}(0.5)
		=
		\left(\frac{1}{1.5}\right)^3
		=
		0.296
		<
		0.779
		=
		e^{-0.25}
		=
		\bar{F}(0.5),
		\]
		which shows that $\bar{F}(t)\leq \bar{G}(t)$ does not hold for all $t\geq0$.
	\end{example}
	
	
	\begin{theorem} \label{Theorem WCRTE for convex order}
		If $X \leq_{icx} Y$, then \(X \leq_{iwcrte} Y \).
	\end{theorem}
	
	\begin{proof}
		Double differentiating $h_{2}(x)$ in Eq.~\eqref{Eq. h2(x) of WCRTE} with respect to $x$, we get
		\[
		h_{2}''(x)
		=
		\frac{1}{\alpha-1}
		\bigl(1-\bar F^{\alpha-1}(x)\bigr)
		+
		x f(x)\bar F^{\alpha-2}(x).
		\]
		
		Since $0<\bar F(x)\le1$, we have $1-\bar F^{\alpha-1}(x)\ge0.$ Also, since $f(x)\ge0$ and $\bar F(x)>0$, it follows that $x f(x)\bar F^{\alpha-2}(x)\ge0.$ Therefore, $h_{2}''(x)\ge0,$ for all $x>t$. Hence, for any given $t$, the function $h_{2}(x)$
		is convex in $x$. The result now follows from the characterization of the convex order, i.e.,
		$X\leq_{icx}Y$
		if and only if
		$\mathbb{E}[h_{2}(X)]\leq \mathbb{E}[h_{2}(Y)]$
		for all convex functions $h_{2}$ for which expectations exist.
	\end{proof}

	The converse implication, however, is not generally true. To illustrate this, we again consider the Weibull and Gompertz distributions introduced in Example~\ref{WCRE convex order counter example}. Figure~\ref{IWCRTE graph} shows that $\xi_{\alpha}^w(X,t)\leq \xi_{\alpha}^w(Y,t)$ for every $t\geq0.$ On the other hand, the curves displayed in Figure~\ref{convex ordering graph} reveal that the inequality
	$\int_{t}^{\infty}\bar{F}(x)\,dx \leq \int_{t}^{\infty}\bar{G}(x)\,dx$
	fails to hold for all $t\geq0;$ indeed, there exists some $t_{0}>0$.
	
	
	
	
	
	

	\begin{figure}[H]
		\centering
		
		\subfloat[Plots of $\int_t^\infty \bar F(x)\,dx$ and $\int_t^\infty \bar G(x)\,dx$.%
		\label{convex ordering graph}]{
			\includegraphics[width=0.45\textwidth]{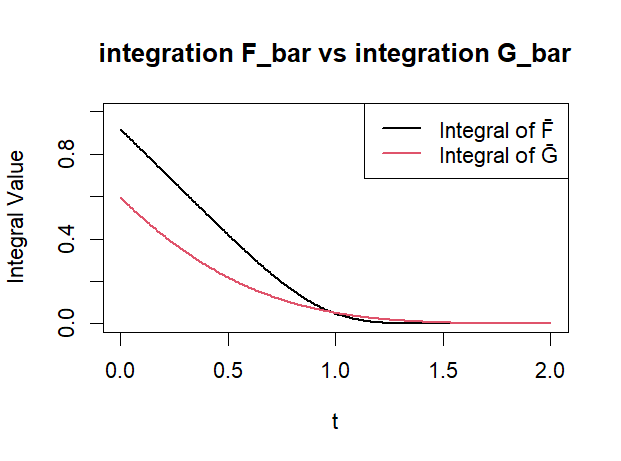}
		}
		\hfill
		\subfloat[Plots of $\xi^{w}(X,t)$ and $\xi^{w}(Y,t)$.%
		\label{IWCRE graph}]{
			\includegraphics[width=0.45\textwidth]{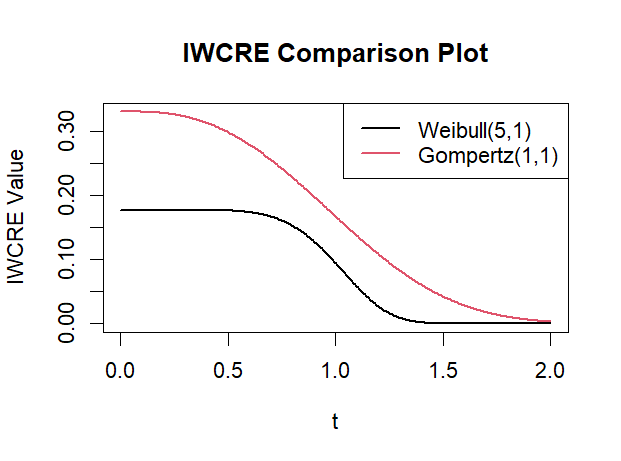}
		}
		
		\vspace{0.5cm}
		
		\subfloat[Plots of $T_{\alpha}(X,t)$ and $T_{\alpha}(Y,t)$.%
		\label{ICRTE graph}]{
			\includegraphics[width=0.45\textwidth]{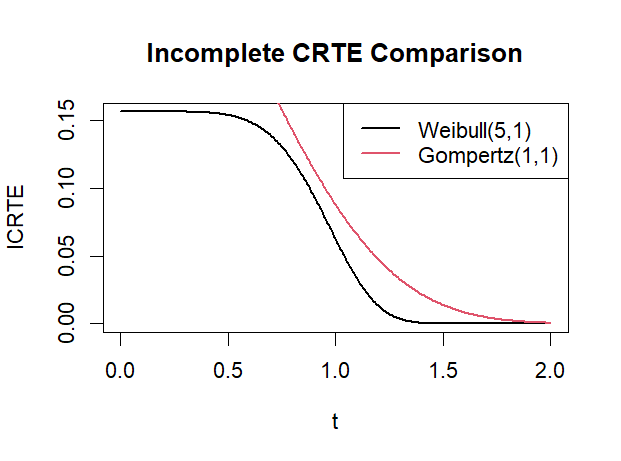}
		}
		\hfill
		\subfloat[Plots of $\xi_{\alpha}^w(X,t)$ and $\xi_{\alpha}^w(Y,t)$.%
		\label{IWCRTE graph}]{
			\includegraphics[width=0.45\textwidth]{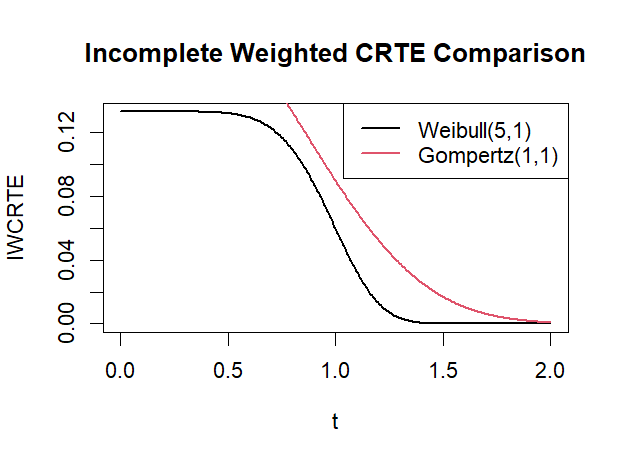}
		}
		
		\caption{Graphical comparison of the considered measures of proposed entropies for the Weibull and Gompertz distributions.}
		\label{fig:comparison_plots}
		
	\end{figure}
	

	\noindent As given in Chakraborty and Nanda (2025), we can express $\xi_{\alpha}^w(X)$ as 
	\begin{equation}
	\xi_{\alpha}^w(X)
	= \frac{1}{2(\alpha-1)}\int_0^\infty x^2 ~ J_{3 \alpha}\left(F(x)\right)~dF(x),
	\qquad 0<\alpha\neq 1.
	\end{equation}
	
	where
	\[
	J_{3\alpha}(u) =(1-\alpha (1-u)^{\alpha-1}).
	\]
	
	\noindent The plug in estimator for the above can be written as 
	
	\[
	\hat{\xi}_{\alpha}^w(X) = \frac{1}{n} \sum_{i=1}^{n} X_{(i)}^2 ~ J_{3\alpha}\left(\frac{i}{n}\right),
	\]

	\subsection{Testing stochastic equality against increasing convex ordering using WCRTE function}

	The objective is to test the following hypothesis:
	Here we take our discrepancy measure as
	\[
	\Delta_{\alpha}^{w}(F,G) = \xi_{\alpha}^w(Y)-\xi_{\alpha}^w(X)
	\]
	The empirical version of $\Delta_{\alpha}^{w}(F,G)$ can be written as
	\[
	\hat{\Delta}_{\alpha}^{w}(F_n,G_m) = \hat{\xi}_{\alpha}^w(Y)-\hat{\xi}_{\alpha}^w(X)
	\]
	
	Asymptotic distribution is 
	\begin{equation*}
	\sqrt{N} (\hat{\Delta}_{\alpha}^w(F_n, G_m)-\Delta_{\alpha}^w(F,G)) \to^d N(0, \sigma^{2}(J_{3 \alpha})),
	\end{equation*}
	where
	\begin{equation*}
	\sigma^2(J_{3\alpha}) = \frac{\sigma_{X}^2(J_{3\alpha})}{\rho}+\frac{\sigma_{Y}^2(J_{3\alpha})}{1-\rho}
	\end{equation*}
	and
	\begin{equation*}
	\sigma_{X}^2(J_{3\alpha}) = 8~\int_{0}^{\infty}\int_{y}^{\infty} x~y~F(y)~\bar{F}(x)~J_{3\alpha}(F(x))~ J_{3\alpha}(F(y))~dx~dy.
	\end{equation*}
	Note that, $\sqrt{N}(\hat{\xi}_{\alpha}^w(X)-\xi_{\alpha}^w(X))\to^dN(0,\sigma_{X}^2(J_{3\alpha}))$, see Chakraborty and Nanda (2025), and an estimator of $\sigma_{X}^2(J_{3\alpha})$ can be written as 
	
	\begin{equation}
	\begin{aligned}
	\hat{\sigma}_{X}^{2}(J_{3\alpha}) 
	&= 8 \sum_{j=1}^{n-2}\sum_{i=j+1}^{n-1} 
	\frac{j}{n} \left(1- \frac{i}{n}\right) 
	J_2\!\left(\frac{i}{n+1}\right) 
	J_2\!\left(\frac{j}{n+1}\right) \\
	&\quad \times \left(X_{(i+1)} - X_{(i)}\right)
	\left(X_{(j+1)} - X_{(j)}\right).
	\end{aligned}
	\end{equation}
	
	Under null hypothesis, 
	\begin{equation*}
	\frac{\hat{\Delta}_{\alpha}^{w}(F_n,G_m)}{\sqrt{\frac{\hat{\sigma}_{X}^{2}(J_{3\alpha})}{n}+\frac{\hat{\sigma}_{Y}^{2}(J_{3\alpha})}{m}}} \to^d N(0,1).
	\end{equation*}

	\section{Monte Carlo Simulation and Comparison}
	\label{section simulation_study}
	
	A comprehensive Monte Carlo simulation study was conducted to investigate the finite-sample performance of the proposed two-sample tests based on CRTE and WCRTE, along with CRE (the test proposed by Zardasht (2015)) and WCRE. The empirical power of the procedures was evaluated using $N_r=10{,}000$ independent replications. In each replication, two independent samples of equal sizes were generated with
	The empirical power of the procedures was evaluated using $N_r=10{,}000$ independent replications. In each replication, two independent samples of equal sizes were generated with $n=m=20,\,50,\,100,\,200,$ so that the total sample size is $N=n+m$. The empirical power was estimated as the proportion of replications in which the corresponding test statistic exceeded the simulated critical threshold. Equal sample sizes were considered throughout the study in order to maintain consistency and ensure a fair comparison with the results reported by Zardasht (2015), since unequal sample sizes could introduce imbalance and potentially affect the comparability of the competing procedures.
	
	In all the power study models, the distributional parameters were chosen so that $\mathbb{E}(X)=\mathbb{E}(Y)$. This creates a deliberately unfavorable setting for discrimination, since the null and alternative distributions already coincide in an important first-order characteristic. Consequently, the tests must distinguish the competing models solely through differences in their ageing behaviour or structural properties rather than through simple mean differences. Therefore, strong empirical power under such restrictive conditions indicates the robustness and sensitivity of the proposed procedures.
	
	The simulation study considered four different distributional frameworks. \begin{enumerate}[label=(\alph*)]
		
		\item \textbf{Weibull Model 1:} 
		In this setting, the random variable $X$ follows a Weibull distribution with shape parameter $10$ and scale parameter $1/\Gamma(1+1/10)$, so that $\mathbb{E}(X)=1$. The variable $Y$ is assumed to follow a Weibull distribution with shape parameter $\theta$ and scale parameter $1/\Gamma(1+1/\theta)$, ensuring that $\mathbb{E}(Y)=1$ as well. The parameter values considered are $\theta=5,6,7,8,9,10$, with the null hypothesis corresponding to $\theta=10$.
		
		\item \textbf{Weibull Model 2:} 
		For this model, $X$ has a Weibull distribution with shape parameter $\theta$ and scale parameter $1/\Gamma(1+1/\theta)$, whereas $Y$ follows an exponential distribution with mean $1$. The parameters are selected such that $\mathbb{E}(X)=\mathbb{E}(Y)=1$. The values $\theta=1,1.3,1.5$ and $1.8$ are examined, and the null hypothesis is associated with $\theta=1$.
		
		\item \textbf{Gamma Model:} 
		Here, $X$ is generated from a Gamma distribution with shape parameter $10$ and scale parameter $1/10$, while $Y$ follows a Gamma distribution with shape parameter $\theta$ and scale parameter $1/\theta$. Under this specification, both distributions have unit expectation. The simulation study considers $\theta=5,6,7,8,9,10$, where the null hypothesis corresponds to $\theta=10$.
		
		\item \textbf{Student's $t$ Model:} 
		In this case, $X$ follows a Student's $t$ distribution with $20$ degrees of freedom, whereas $Y$ follows a Student's $t$ distribution with $v$ degrees of freedom. The parameter values taken into consideration are $v=2,5,10,15$ and $20$, with the null hypothesis corresponding to $v=20$.
		
	\end{enumerate}
	
	The simulation results demonstrate several important features of the proposed procedures. For all models and sample sizes, the empirical rejection probabilities approach the nominal significance level of $5\%$ as the model parameters move closer to the null configuration, thereby confirming satisfactory control of the Type~I error rate. Furthermore, the empirical power increases monotonically with the sample size, which is expected since larger samples provide greater discriminatory ability between the null and alternative distributions. The influence of the Tsallis parameter $\alpha$ was also examined under both CRTE and WCRTE statistics. Across the majority of the considered alternatives, the highest empirical power was achieved at $\alpha=1.3$ for the CRTE-based test and at $\alpha=1.5$ for the WCRTE-based test. An exception was observed for the Student's $t$ alternatives under WCRTE, where $\alpha=1.1$ yielded slightly better performance.

	\begin{table}[H]
		\centering
		\caption{Empirical power comparison for Weibull Model 1 at 5\% nominal level}
		\begin{tabular}{c c c c c c}
			\hline
			$n=m$ & $\theta$ & CRE & WCRE & CRTE  & WCRTE \\
			$\downarrow$ & $\downarrow$ & Zardasht (2015) & & ($\alpha=1.3$) & ($\alpha=1.5$) \\
			\hline
			
			\multirow{6}{*}{20}
			& 5  & 0.829 & 0.874 & 0.912 & 0.984 \\
			& 6  & 0.641 & 0.653 & 0.739 & 0.926 \\
			& 7  & 0.395 & 0.438 & 0.483 & 0.750 \\
			& 8  & 0.229  & 0.200 & 0.243  & 0.566 \\
			&  9 & 0.103 &0.110  &  0.124  & 0.341 \\
			& 10 & 0.052 & 0.048 & 0.051 & 0.056 \\
			\hline
			
			\multirow{6}{*}{50}
			& 5  & 0.996 & 1.000 & 0.999 & 1.000 \\
			& 6  & 0.949 & 0.967 & 0.969 & 0.998 \\
			& 7  & 0.720 & 0.748 & 0.773 & 0.952 \\
			&  8 & 0.413 & 0.412 & 0.419 & 0.754 \\
			&  9 & 0.141 & 0.129 & 0.167 & 0.453 \\
			& 10 & 0.048 & 0.046 & 0.048 & 0.048 \\
			\hline
			
			\multirow{6}{*}{100}
			& 5  & 1.000 & 1.000 & 1.000 & 1.000 \\
			&  6 & 0.998  & 1.000 & 1.000 & 1.000 \\
			& 7  & 0.941 & 0.983 & 0.988 & 0.999 \\
			&  8 & 0.653 & 0.728 & 0.696 & 0.885 \\
			&  9 & 0.243 & 0.250 & 0.251 & 0.574 \\
			& 10 & 0.046 & 0.054 & 0.054 & 0.049 \\
			\hline
			
			\multirow{6}{*}{200}
			& 5  & 1.000 & 1.000 & 1.000 & 1.000 \\
			& 6  & 1.000 & 1.000 & 1.000 & 1.000 \\
			& 7  & 0.998 & 0.999 & 0.992 & 1.000 \\
			& 8 & 0.920 & 0.938 & 0.936 & 0.989  \\
			& 9  & 0.438  & 0.446 & 0.451 & 0.723 \\
			& 10 & 0.049 & 0.047 & 0.049 & 0.053 \\
			\hline
		\end{tabular}
	\end{table}


	\begin{table}[H]
		\centering
		\caption{Empirical power comparison for Weibull Model 2 (at 5\% nominal level of significance)}
		\begin{tabular}{c|c|c|c|c|c}
			\hline
			$n=m$ & $\theta$ & CRE & WCRE & CRTE  & WCRTE \\
			$\downarrow$ & $\downarrow$ & Zardasht (2015) & & ($\alpha=1.3$) & ($\alpha=1.5$) \\
			\hline
			
			\multirow{4}{*}{20}
			& 1   & 0.058 & 0.056 & 0.052 & 0.056 \\
			& 1.3 & 0.124 & 0.083 & 0.089 & 0.191 \\
			& 1.5 & 0.171 & 0.149 & 0.148 & 0.214 \\
			& 1.8 & 0.279 & 0.232 & 0.258 & 0.384 \\
			
			\hline
			\multirow{4}{*}{50}
			& 1   & 0.061 & 0.055 & 0.051 & 0.055 \\
			& 1.3 & 0.241 & 0.164 & 0.252 & 0.301 \\
			& 1.5 & 0.527 & 0.319 & 0.624 & 0.532 \\
			& 1.8 & 0.983 & 0.559 & 0.718 & 0.916 \\
			
			\hline
			\multirow{4}{*}{100}
			& 1   & 0.035 & 0.043 & 0.051 & 0.460 \\
			& 1.3 & 0.425& 0.342 & 0.672 & 0.500 \\
			& 1.5 & 0.741 & 0.600 & 0.752 & 0.787 \\
			& 1.8 & 0.998 & 0.872 & 0.953 & 1.000 \\
			
			\hline
			\multirow{4}{*}{200}
			& 1   & 0.053 & 0.049 & 0.052 & 0.049 \\
			& 1.3 & 0.426 & 0.542 & 0.583 & 0.621 \\
			& 1.5 & 0.965 & 0.872 & 0.944 & 0.985 \\
			& 1.8 & 0.999 & 0.994 & 1.000 & 1.000 \\
			\hline
		\end{tabular} 
	\end{table}

	\begin{table}[H]
		\centering
		\caption{Empirical power comparison for Gamma Model (at 5\% nominal level of significance)}
		\begin{tabular}{c|c|c|c|c|c}
			\hline
			$n=m$ & $\theta$ & CRE & WCRE & CRTE  & WCRTE \\
			$\downarrow$ & $\downarrow$ & Zardasht (2015) & & ($\alpha=1.3$) & ($\alpha=1.5$) \\
			\hline
			
			\
			\multirow{6}{*}{20}
			& 5 & 0.421 & 0.274 & 0.268 & 0.236 \\
			& 6 & 0.285 & 0.179 & 0.222 & 0.158 \\
			& 7 & 0.204 & 0.114 & 0.147 & 0.122 \\
			& 8 & 0.118 & 0.096 & 0.101 & 0.083 \\
			& 9 & 0.085 & 0.082 & 0.074 & 0.068 \\
			& 10 & 0.057 & 0.055 & 0.056 & 0.052 \\
			\hline

			\multirow{6}{*}{50}
			& 5  & 0.640 & 0.443 & 0.636 & 0.461 \\
			& 6 & 0.466 & 0.307 & 0.467 & 0.301 \\
			& 7  & 0.280 & 0.213 & 0.295 & 0.223 \\
			& 8  & 0.179 & 0.123 & 0.188 & 0.130 \\
			& 9 & 0.106 & 0.065 & 0.106 & 0.075 \\
			& 10 & 0.056 & 0.056 & 0.047 & 0.053 \\
			\hline

			\multirow{6}{*}{100}
			& 5  & 0.859 & 0.788 & 0.904 & 0.751 \\
			& 6 & 0.647 & 0.450 & 0.664 & 0.523 \\
			& 7  & 0.439 & 0.257 & 0.504 & 0.342 \\
			& 8 & 0.207 & 0.144 & 0.244  & 0.191 \\
			& 9 & 0.104 & 0.076 & 0.132 & 0.122 \\
			& 10 & 0.051 & 0.046 & 0.049 & 0.050 \\
			\hline

			\multirow{6}{*}{200}
			& 5  & 0.989 & 0.932 & 0.994 & 0.955 \\
			& 6 &0.876 & 0.802 & 0.905 & 0.811 \\
			& 7  & 0.654 & 0.548 & 0.696 & 0.558 \\
			& 8 &0.385 & 0.302 & 0.406 & 0.331 \\
			& 9 &0.152 & 0.167 & 0.157 & 0.154 \\
			& 10 & 0.048 & 0.047 & 0.052 & 0.048 \\
			\hline
		\end{tabular} 
	\end{table}

	\begin{table}[H]
		\centering
		\caption{Empirical power comparison for Student's $t$ Model (at 5\% nominal level of significance)}
		\begin{tabular}{c|c|c|c|c|c}
			\hline
			$n=m$ & $\theta$ & CRE & WCRE & CRTE  & WCRTE \\
			$\downarrow$ & $\downarrow$ & Zardasht (2015) & & ($\alpha=1.3$) & ($\alpha=1.5$) \\
			\hline
			
			\
			\multirow{5}{*}{20}
			& 2  & 0.523 & 0.194 & 0.491 & 0.341 \\
			& 5  & 0.139 & 0.094 & 0.161 & 0.135 \\
			& 10 & 0.072 & 0.053 & 0.082 & 0.065 \\
			& 15 & 0.053 & 0.045 & 0.054 & 0.053\\
			& 20 & 0.048 & 0.049 & 0.052 & 0.048 \\
			\hline
			
			\multirow{5}{*}{50}
			& 2  & 0.832 & 0.330 & 0.867 & 0.465 \\
			& 5  & 0.240 & 0.149 & 0.281 & 0.202 \\
			& 10 & 0.071 & 0.137 & 0.107 & 0.087 \\
			& 15 & 0.059 & 0.069 & 0.077 & 0.069 \\
			& 20 & 0.047 & 0.054 & 0.048 & 0.055 \\
			\hline
			
			\multirow{5}{*}{100}
			& 2  & 0.930 & 0.358 & 0.987 & 0.622 \\
			& 5  & 0.374 & 0.137 & 0.434 & 0.283 \\
			& 10 & 0.116 & 0.075 & 0.124 & 0.114 \\
			& 15 & 0.062 & 0.069 & 0.063 & 0.072 \\
			& 20 & 0.054 & 0.054 & 0.052 & 0.053 \\
			\hline
			
			\multirow{5}{*}{200}
			& 2  & 0.992 & 0.520 & 1.000 & 0.687 \\
			& 5  & 0.642 & 0.227 & 0.658 & 0.341 \\
			& 10 & 0.176 & 0.088 & 0.191 & 0.111 \\
			& 15 & 0.104 & 0.070 & 0.075 & 0.072 \\
			& 20 & 0.049 & 0.058 & 0.051 & 0.049 \\
			\hline
		\end{tabular} 
	\end{table}
	
	Overall, the proposed Tsallis entropy-based procedures exhibit strong and competitive performance across a broad range of alternatives. In particular, the WCRTE-based test consistently outperforms the procedure of Zardasht (2015) for both Weibull models, whereas the CRTE-based test demonstrates superior performance for the Gamma and Student's $t$ models. These findings indicate that the proposed methodology provides a flexible and effective framework for two-sample discrimination problems, while also highlighting the importance of selecting an appropriate value of the tuning \mbox{parameter} $\alpha$ to achieve optimal testing performance.
	
	
	\section{Conclusion}
	\label{section Conclusion}
	
	In this paper, we developed entropy-based stochastic orderings corresponding to IWCRE, ICRTE, and IWCRTE, and studied their relationships with certain existing stochastic orderings. Based on these characterizations, nonparametric tests for stochastic equality against ordered alternatives were proposed and their asymptotic properties were established. Extensive simulation studies under different distributional settings were carried out to examine the finite-sample behaviour of the proposed procedures.
	
	The simulation results indicate that the proposed Tsallis entropy-based tests perform competitively across a wide range of alternatives and, in several situations, outperform the test proposed by Zardasht (2015). In particular, the WCRTE- and CRTE-based procedures exhibited improved discriminatory performance under the considered Weibull, Gamma, and Student's $t$ models.

	The present work is developed for complete data only. However, in many practical applications, lifetime observations are often subject to censoring, resulting in incomplete information. Extending the proposed methodologies to censored data settings therefore constitutes an important direction for future research. Another possible avenue for further investigation is the study of other stochastic orderings, such as the dilation order, together with the development of related testing procedures based on these and other entropy measures.


\end{document}